\documentclass[12pt,leqno,twoside]{amsart}

\usepackage[latin1]{inputenc}
\usepackage[T1]{fontenc}
\usepackage{amstext,amsmath,amscd, bezier,indentfirst,amsthm,amsgen,enumerate, geometry}
\usepackage[all,knot,arc,import,poly]{xy}
\usepackage{amsfonts,color}
\usepackage{amssymb}
\usepackage{latexsym}
\usepackage{epsfig}
\usepackage{graphicx}
\usepackage{srcltx}
\usepackage{enumitem}

\theoremstyle{plain}
\newtheorem{theorem}{Theorem}[section]
\newtheorem{lemma}[theorem]{Lemma}
\newtheorem{proposition}[theorem]{Proposition}

\newtheorem{conjecture}[theorem]{Conjecture}

\theoremstyle{definition}
\newtheorem{definition}[theorem]{Definition}

\theoremstyle{remark}
\newtheorem{remark}[theorem]{\sc Remark}
\newtheorem{example}[theorem]{\sc Example}

\def\bC{{\mathbb C}}
\def\bK{{\mathbb K}}
\def\bN{{\mathbb N}}
\def\bP{{\mathbb P}}

\def\bR{{\mathbb R}}

\def\Si{{ \mathcal{S}_{\infty}}}

\def\Bi{{ \mathcal{B}_{\infty}}}

\def\ity{\infty}
\def\Arc{{\rm Arc}}
\def\rD{{\rm D}}

\def\grad{{\rm grad}}
\def\Sing{{\rm Sing}}

\def\rank{{\rm rank\ }}

\def\const.{{\rm const.}}

\def\m{{\setminus}}

\begin{document}
\title[Bifurcation values at infinity of real polynomials]{Detecting bifurcation values at infinity of real polynomials}
\date{\today}
\author{Luis Renato G. Dias}
\address{Faculdade de Matem\'atica, Universidade Federal de Uberl\^andia, Av. Jo\~ao Naves de \'Avila 2121, 1F-153 - CEP: 38408-100, Uberl\^andia, Brazil}
\email{lrgdias@famat.ufu.br}
\author{Mihai Tib\u ar}
\address{Math\'ematiques, Laboratoire Paul Painlev\'e, Universit\'e Lille 1,
59655 Villeneuve d'Ascq, France.}
\email{tibar@math.univ-lille1.fr}
\subjclass[2010]{}

%\tableofcontents


\thanks{The authors acknowledge the support of the USP-COFECUB Uc
Ma 133/12 project.}

\begin{abstract}
We present a new approach for estimating the set of bifurcation values at infinity. This yields a significant shrinking of the number of coefficients in the recent algorithm introduced by Jelonek and Kurdyka for reaching 
critical values at infinity by rational arcs.
\end{abstract}

\maketitle

\section{Introduction}

 The \textit{bifurcation locus} of a polynomial mapping $f\colon \bK^n \to \bK^p$, where $\bK = \bR$ or $\bC$ and $n \geq p$ is the smallest subset $B(f) \subset \bK^p$ such that $f$ is a locally trivial $C^\ity$-fibration over $\bK^p\setminus B(f)$.  
It is well known that $B(f) = f(\Sing f)\cup \Bi(f)$, where $\Bi(f)$ denotes the set of \textit{bifurcation values at infinity} (Definition \ref{d:triv_at_inf}). A simple example is $f(x,y)= x + x^2y$, where  we have $\Sing f = \emptyset$ and $\Bi(f) = \{ 0\}$.

While the set of critical values $f(\Sing f)$ is relatively well understood, the other set $\Bi(f)$ is still mysterious. In case $p=1$ the bifurcation set $B(f)$ is finite, as proved by Thom \cite{Th}, see also \cite{Ph,Ve}.
However,  one can precisely detect the bifurcation set $B(f)$ only in case $p=1$ and $n=2$ by using several types of (equivalent) tests, see \cite{Su}, \cite{HL}, \cite{Ti-reg}, \cite{Dur}, \cite{Ti-book} over $\bC$  and \cite{TZ}, \cite{CP} over $\bR$. 

For  $p=1$ and more than two variables one can only estimate $B(f)$ by ``reasonably good'' supersets $A(f) \supset B(f)$ of the form $A(f)=f(\Sing f)\cup A_{\infty}(f)$, where $A_{\infty}(f)$ is a finite set which depends on the choice of some regularity condition at infinity: \emph{tameness} \cite{Br1},  \emph{Malgrange regularity} \cite{Pa}, \emph{$\rho$-regularity} \cite{NZ}, \cite{Ti-cras},  \emph{$t$-regularity}  \cite{ST}, \cite{Ti-reg}.  

 In case $p>1$ it is known that $\Bi(f)$ is contained in a one codimensional semi-algebraic subset of $\bR^p$, or an algebraic subset of $\bC^p$, respectively, as proved by Kurdyka, Orro and Simon \cite{KOS}. Sharper such subsets have been obtained in \cite{DRT}.

This paper addresses the problem of estimating the bifurcation locus at infinity $\Bi(f)$ and is motivated by the recent algorithm presented by 
Jelonek and Kurdyka \cite{JK} for finding the set of values $K_\ity(f)$ for which the \emph{Malgrange condition} at infinity fails (Definition \ref{d:ra-ity}). Their algorithm applies to a space $AV(f)$ of rational arcs in $\bR^n$, the coefficients of which are solutions of a certain set of equations.  

We present here the construction of a different space of rational arcs $\Arc(f)$ to which we attach the same set of equations for  the coefficients, and thus the same algorithm can be run. Our main results, Theorem \ref{t:main} together with Theorems \ref{t:asympSard} and \ref{p:rho-rab}, show that the resulting subset of asymptotic arcs $\Arc_\ity(f)$ detects a certain set of values that includes $\Bi(f)$ and is included in $K_\ity(f)$. Consequently, the Jelonek-Kurdyka algorithm applied to our $\Arc(f)$ will produce a (smaller) set of values including $\Bi(f)$ in considerably shorter time since the number of coefficients of the rational arcs is drastically reduced, namely:
 \[ \dim \Arc(f) =n(1+ d^{n}) \mbox{\ \  versus \ \  }  \dim AV(f)= n(2+ d(d+1)^n(d^n+2)^{n-1}). \]

Our result is also relevant for optimisation and complexity problems since bifurcation values at infinity appear for instance in the optimization of real polynomials, e.g. \cite{HP2}, \cite{Sa}.

%For problem (1) there are several contributions in case of complex polynomials, for 2 variables %e.g. in \cite{GP}, \cite{Gw}, \cite{JT}, for  $n$ variables in \cite{JK-crelle}.  Here we find a %sharp bound in the real setting directly. This is in terms of the degree of $f$.

\bigskip

%%%%%%%%%%%%%%%%%%%%%%%%%%%%%%%

\section{Regularity conditions at infinity}\label{s:reg-ity}

\subsection{Bifurcation values at infinity}
We start by recalling the basic definitions after \cite{Ti-reg}, \cite{Ti-book}, \cite{DRT}, eventually adapting them to any $n\ge p\ge 1$.

\begin{definition}\label{d:bif}\label{d:triv_at_inf}
Let $f\colon\bR^n\to\bR^p$ be a polynomial mapping, where $n\geq p$.
We say that $t_0\in\bR^p$ is a \textit{typical value}  of $f$ if there exists a disk $D\subset \bR^p$ centered at $t_0$ such that the restriction $f_{|}\colon f^{-1}(D) \to D$ is a locally trivial $C^\ity$-fibration. Otherwise we say that $t_0$ is an \textit{bifurcation value} (or atypical value). We denote by $B(f)$ the \emph{bifurcation locus}, i.e. set of bifurcation values of $f$.

 We say that $f$ is {\em topologically trivial at infinity at $t_0\in \bR^p$} if there exists
a compact set $K\subset \bR^n$ and a disk $D \subset \bR^p$ centered at $t_0$ such that the restriction
$f_| :  f^{-1}(D)\setminus K \to D$
 is a locally trivial $C^\ity$-fibration. 
Otherwise we say that $t_0$ is a {\em bifurcation value  at infinity of} $f$. We denote by $\Bi(f)$ the set of bifurcation values at infinity of $f$.
\end{definition}

As we have claimed in the Introduction, there is no general characterisation\footnote{Still, one can treat particular cases when ``singularities at infinity'' are isolated in a certain sense, e.g. \cite{ST}, \cite{Pa}, \cite{Ti-book}, \cite{TT}.} of the bifurcation locus besides the setting $n=2$ and $p=1$.   
In case $n>2$ one uses regularity conditions at infinity in order to control the topological triviality. We work here with the $\rho$-regularity and with the Malgrange-Kuo-Rabier condition. 

\subsection{The $\rho$-regularity}\label{ss:rho}

 We adapt the definition of $\rho$-regularity from \cite{Ti-reg} to Euclidean spheres centered at any point of $\bR^n$. This allows us to define a new set, denoted here by $S_{\infty}(f)$, which produces a sharper estimation of $B(f)$.

Let $a=(a_1,\ldots,a_n)\in\bR^n$ and let  $\rho_a\colon\bR^n \to \bR_{\geq 0}$ be the Euclidian distance function to $a$, i.e.  $\rho_{a}(x)=(x_1-a_1)^2+\ldots+(x_n-a_n)^2$. 

\begin{definition}[Milnor set at infinity]\label{d:milnor}
Let $f\colon\bR^n\to\bR^p$ be a polynomial mapping,  where $n\geq p$. We fix $a\in\bR^n$. We call \emph{Milnor set of $(f,\rho_a)$} the critical set of the mapping $(f, \rho_a)\colon\bR^n\to\bR^{p+1}$ and  denote it by $M_a(f)$. 
\end{definition}
 In case $p=1$ we need the following statement, which was noticed in \cite{HP} (see also \cite[Lemma 2.2]{Du}). We provide a proof and use some details of it in the proof of Theorem \ref{t:main}.

\begin{lemma}\label{p:gen-a} Let $f\colon\bR^n\to\bR$ be a polynomial mapping. There exists an open dense subset $\Omega_f\subset \bR^n$ such that, for every $a\in \Omega_f$,   $M_a(f)\setminus\Sing f$ is either a non-singular curve or it is empty. 
\end{lemma}

\begin{proof}
We claim that the following semi-algebraic set:
\begin{equation}
Z:=\{(x,a)\in\bR^n\times\bR^n\mid x\in M_a(f)\setminus\Sing f\}.
\end{equation}
is a smooth $(n+1)$-dimensional manifold. For some fixed $(x_0,a_0)\in Z$,  there is $1\leq i\leq n$ such that $\frac{\partial f}{\partial x_i}(x_0)\neq 0$ and moreover, since $\Sing f$ is closed, there exists an open set $U\subset \bR^n$ such that $\frac{\partial f}{\partial x_i}(x)\neq 0, \forall x\in U$. For $1\leq j\leq n$, $j\neq i$, we set: 
\begin{equation}\label{eq:m-j}
m_j(x,a):=\frac{\partial f}{\partial x_i}(x)(x_j-a_j)-\frac{\partial f}{\partial x_j}(x)(x_i-a_i).
\end{equation} 
We have $Z\cap(U\times \bR^n)=\{(x,a)\in U\times\bR^n\mid m_j(x,a)=0; \mbox{ $1\leq j\leq n$,  $j\neq i$ }\}$. Let  $\varphi\colon U\times \bR^n\to\bR^{n-1}$ be the mapping $\varphi = (m_1, \ldots , m_n)$, where $m_i$ is missing. By its definition we have $\varphi^{-1}(0)=Z\cap (U\times\bR^n)$. Let us notice that the rank of the gradient matrix $D\varphi$ is $n-1$ at any point of $U\times\bR^n$. Indeed,  the minor $(\frac{\partial m_l}{\partial a_k})_{1\leq l, k\leq n}^{l, k\neq i}(x,a)$ is a diagonal matrix of which the entries on the diagonal are all equal to  $- \frac{\partial f}{\partial x_i}(x)$, hence non-zero. This shows that $Z$ is a manifold of dimension $n+1$. 

We next consider the projection $\tau\colon Z\to\bR^n$, $\tau(x,a)=a$. Thus, $\tau^{-1}(a)=(M_a(f)\setminus\Sing f)\times\{a\}$. By Sard's Theorem, we conclude that,  for almost all $a\in\bR^n$, $\tau^{-1}(a) = (M_a(f)\setminus\Sing f)\times\{a\}\cong (M_a(f)\setminus\Sing f)$ is a smooth curve or that it is an empty set. 
\end{proof}

\begin{definition}[$\rho_a$-regularity at infinity]\label{d:rho-reg}
Let $f\colon\bR^n\to\bR^p$ be a polynomial mapping,  where $n\geq p$. Let $a\in\bR^n$. We call:
\begin{equation}
 S_a(f):=\{t_0\in\bR^p\mid\exists \{ x_j\}_{j\in \bN}\subset M_a(f), \lim_{j\to\infty}\| x_j\|=\infty\mbox{ and }\lim_{j\to\infty}f(x_j)=t_0\}.
\end{equation}
the \emph{set of asymptotic $\rho_a$-nonregular values}. If $t_0\notin S_a(f)$ we say that $t_0$ is \textit{$\rho_a$-regular at infinity}. We set  
$\Si (f):= \bigcap_{a\in \bR^n} S_{a}(f)$.

\end{definition}

The above condition is a ``Milnor type'' condition that controls the transversality of the fibres of $f$ to the  spheres centered at $a\in\bR^n$. \footnote{For complex polynomial functions, transversality to large spheres  was used in \cite[p.229]{Br1}, in \cite{NZ} where it is called {\em M-tameness}, later in \cite{ST}, \cite{Dur} etc. In the real setting, it was used in \cite{Ti-reg},  later in \cite{HP} etc.}

We derive the following result about $\Si (f)$ from  \cite[Theorem 5.7,  Proposition 6.4]{DRT} where it was stated for $S_0(f)$:

\begin{theorem}\label{t:asympSard}
 Let $f\colon\bR^n\to\bR^p$ be a polynomial mapping, where $n\geq p\ge 1$. Then
 $\Bi(f) \subset \Si (f)$.

Moreover:
\begin{enumerate}
\rm \item \it 
 The sets $\Si (f)$ and $f(\Sing f)\cup \Si (f)$  are closed sets. 

\rm \item  \it  
For any $a\in\bR^n$, $S_a(f)$ and $f(\Sing f)\cup S_{a}(f)$ are semi-algebraic sets of dimension $\le p-1$.

\end{enumerate}
\end{theorem}

\begin{proof}

To prove the inclusion $\Bi(f) \subset \Si (f)$, let $t_0\in\bR^p\m  \Si (f)$. It follows that there exists $a\in\bR^n$  such that $t_0\notin  S_a(f)$.  By \cite[Prop. 6.4]{DRT} we obtain that $f$ is topologically trivial at infinity at $t_0\in \bR^p$, after Definition \ref{d:triv_at_inf}, where the later is an open set (as shown at (a) below). 
A completely similar reasoning shows the inclusion $B(f) \subset f(\Sing f) \cup \Si (f)$, thus ending the proof of the first part of our statement.

(a). The proof of \cite[Theorem 5.7(a), p. 337]{DRT} yields that $S_{0}(f)$ and $S_{0}(f)\cup f(\Sing f)$ are closed sets. The same proof holds true when the point 0 is replaced by any other point $a\in \bR^n$. This implies that $\Si (f)=\bigcap_{a\in\bR^n} S_{a}(f)$ and $ \bigcap_{a\in\bR^n} A_{\rho_a}(f)$ are closed sets.
 
The claim (b) was proved as \cite[Theorem 5.7(b)]{DRT} for $S_0(f)$ and the same proof holds for any $a\in\bR^n$.
\end{proof}

\begin{remark}\label{r:main}
Since  $M_a(f)$ is semi-algebraic,  for any value $c\in S_a(f)$ there exist paths $\phi : ]0,\varepsilon [ \to M_a(f)\subset \bR^n$ such that  $\lim_{t\to 0}\|\phi(t)\|=\infty$ and $\lim_{t\to 0}f(\phi(t))= c$. This follows from the Curve Selection Lemma at infinity, as remarked in \cite{DRT} and \cite{CDTT}. 
\end{remark}

%%%%%%%%%%%%%%%%%%%%%%%%%%%

\subsection{Relation with the Malgrange-Kuo-Rabier condition}\label{ss:kuo}

Jelonek an Kurdyka  \cite{JK-crelle} gave an estimation for $B(f)$ by using the notion of asymptotic critical values of $f$. This is based on the Malgrange-Kuo-Rabier regularity at infinity.  We have shown in \cite{Ti-cras} and \cite{DRT} that this condition implies the $\rho$-regularity at infinity, where $\rho$ denotes the Euclidean distance, but that they are not the same and therefore 
 the associated sets of ``critical values at infinity'' may be different (see Example \ref{ex:pz}).

\begin{definition}\label{d:ra-ity}\cite{Ra}
Let $f\colon\bR^n\to\bR^p$ be a polynomial mapping, where $n\geq p$. Let $\rD f(x)$ be the Jacobian matrix of $f$ at $x$.  Let
\begin{eqnarray}\label{eq:ra-ity}
K_{\ity}(f) & := & \{t\in\bR^{p}\mid \exists \{ x_{j}\}_{j\in \bN} \subset \bR^{n}, \lim_{j\to\ity}\|
x_j\|=\ity, \\ \nonumber
 &  &  \lim_{j\to\ity}f(x_j)= t \mathrm{\ and\ }\lim_{j\to\ity}\|x_j\|\nu(\rD f (x_j))=0\},
\end{eqnarray}
where $\nu(A):=\inf_{\|\varphi\|=1}\|A^*(\varphi)\|$, for 
a linear mapping $A\colon\bR^n\to\bR^p$ and its adjoint $A^*\colon(\bR^p)^*\to(\bR^n)^*$.
We call $K_{\ity}(f)$  the set of \emph{asymptotic critical values of $f$}.
\end{definition}
Note that in case $p=1$ the last limit in \eqref{eq:ra-ity} amounts to $\lim_{j\to\ity}\|x_j\| \|\grad f (x_j))\| =0$ thus the above definition recovers the definition of \emph{Malgrange non-regular values at infinity} (cf \cite{Pa}, \cite{ST}, \cite{Ti-cras} etc).

The next key result shows the inclusion\footnote{The particular inclusion $S_{0}(f)\subset K_{\infty}(f)$ was proved in  \cite[Proposition 2.4]{CDTT}.} $S_a(f)\subset K_{\infty}(f)$ for any $a\in \bR^n$ (which may be strict, see Example \ref{ex:pz}).  Moreover, it shows not only that all the values of  $S_a(f)$ may be detected by paths $\phi$ like in the above Remark \ref{r:main}, but that the same paths $\phi$ verify the conditions \eqref{eq:ra-ity} in the definition of $K_\ity(f)$. This will be a key ingredient in the proof of Theorem \ref{t:main}.

\begin{theorem}\label{p:rho-rab}
 Let $f=(f_1,\ldots,f_p)\colon\bR^n\to\bR^p$ be a polynomial  mapping, where $n> p$. Let $\phi : ]0,\varepsilon [ \to M_a(f)\subset \bR^n$ be an analytic path such that  $\lim_{t\to 0}\|\phi(t)\|=\infty$ and $\lim_{t\to 0}f(\phi(t))= \mathrm{c}$. Then $\lim_{t\to 0}\|\phi(t)\| \nu(\rD f (\phi(t)))=0$.

In particular
 $S_a(f)\subset K_{\infty}(f)$ and $\Si (f)\subset K_{\infty}(f)$.
\end{theorem}

\begin{proof}

Let $\mathrm{c}=(c_1,\ldots,c_p)\in S_a(f)$. 
By the definition of $M_a(f)$, we have $\phi(t)\in  M_a(f)$ if and only if $\rank \rD (f, \rho_a)(\phi(t))< p+1$.
It follows that 
 there exist real coefficients  $\lambda(t), b_1(t),\ldots, b_p(t)$ which are actually analytic functions of $t$ and not all equal to zero for the same $t$,  such that:
\begin{equation} \label{eq:rho-kos2}
 \lambda(t)\cdot (\phi_1(t) -a_1,\ldots,\phi_n(t)-a_n)=b_1(t)\frac{\partial f_1}{\partial x}(\phi(t))+\cdots+b_p(t)\frac{\partial f_p}{\partial x}(\phi(t)),
\end{equation}
where $\frac{\partial f_i}{\partial x}(\phi(t)) := \left(\frac{\partial f_i}{\partial x_1}(\phi(t)),\ldots,\frac{\partial f_i}{\partial x_n}(\phi(t))\right)$ for $i=1,\ldots,p$.

Let us denote $\hat \lambda(t):=\frac{\lambda(t)}{\|b(t)\|}$ and $\hat b(t):=\frac{b(t)}{\|b(t)\|}$, where
 $b(t)=(b_1(t),\ldots,b_p(t))$. 
Since $b(t)\neq 0$ $\forall t\in ]0,\varepsilon [$ for small enough $\varepsilon$, the equality \eqref{eq:rho-kos2} writes:

\begin{equation}\label{eq:rho-kos3}
 \sum_{i=1}^p  \hat b_i(t)\frac{\partial f_i}{\partial x}(\phi(t))=\hat \lambda(t)(\phi(t)- a).
\end{equation}
From this we obtain:
\begin{equation} \label{eq:rho-kos4}
\sum_{i=1}^p \hat b_i(t)\frac{\mathrm{d}}{\mathrm{d}t} f_i(\phi(t))=\left\langle\sum_{i=1}^p  \hat b_i(t)\frac{\partial f_i}{\partial x}(\phi(t)),\ \phi'(t)\right\rangle\\ =\frac{1}{2}\hat \lambda(t)\frac{\mathrm{d}}{\mathrm{dt}}\|\phi(t) - a\|^2.
\end{equation} 

We shall denote by $\mathrm{ord_t}$ the order at $0$ of some analytic parametrisation.  The condition $\lim_{t\to 0}f_i(\phi(t))=c_i$ implies that $\mathrm{ord_t}\left(\frac{\mathrm{d}}{\mathrm{dt}}f_i(\phi(t))\right)\geq 0$, $i=1,\ldots, p$, and the condition $\| \hat b(t) \| =1$ implies that the order of the first sum in \eqref{eq:rho-kos4} is also non-negative. We may therefore derive from \eqref{eq:rho-kos4}:
\begin{equation}\label{eq:rho-kos5}
 0\leq\mathrm{ord_t}\left(\hat \lambda(t)\frac{\mathrm{d}}{\mathrm{dt}}\|\phi(t)-a\|^2\right)<\mathrm{ord_t}\left(\hat \lambda(t)\|\phi(t)-a\|^2\right).
\end{equation}

By taking norms in \eqref{eq:rho-kos2} and multiplying by $\|\phi(t)-a\|$ we get:
\begin{equation}%\label{eq:rho-kos6}
 \mathrm{ord_t}\left(\|\phi(t)-a\| \| \hat b_1(t)\frac{\partial f_1}{\partial x}(\phi(t))+\ldots+ \hat b_p(t)\frac{\partial f_p}{\partial x}(\phi(t))\|\right)=\mathrm{ord_t}\left(|\hat \lambda(t)|\|\phi(t)-a\|^2\right),
\end{equation}
which is positive by (\ref{eq:rho-kos5}). This implies:
\[\lim_{t\to 0}\|\phi(t)-a \| \| \hat b_1(t)\frac{\partial f_1}{\partial x}(\phi(t))+\ldots+ \hat b_p(t)\frac{\partial f_p}{\partial x}(\phi(t))\|=0,\]
which, in turn, implies $\lim_{t\to 0}\|\phi(t)-a\|\nu(\mathrm{d}F(\phi(t)))=0$. Since $\mathrm{ord_t} \|\phi(t)-a\|=\mathrm{ord_t} \|\phi(t)-a\| <0$, we get $\lim_{t\to 0}\|\phi(t)\|\nu(\mathrm{d}F(\phi(t)))=0$, which shows that $\mathrm{c}\in K_{\infty}(F)$.

In order to show the claimed inclusions $S_a(f)\subset K_{\infty}(f)$ and $\Si (f)\subset K_{\infty}(f)$ we use Remark \ref{r:main} and the above proof. 
\end{proof}

The inclusion $\Si (f)\subset K_{\infty}(f)$ may be strict, as showed by the next example.

\begin{example}\label{ex:pz}\cite{PZ}
 The polynomials $f_{nq}:\mathbb{K}^{3}\rightarrow\mathbb{K}$,
$f_{nq}(x,y,z):=x-3x^{2n+1}y^{2q}+2x^{3n+1}y^{3q}+yz$,
where $n,q\in\mathbb{N}\setminus\{0\}$. We have  
  $S_0(f_{nq})=\emptyset$. For $\bK=\bC$, it is shown in \cite{PZ} that $f_{nq}$ satisfies Malgrange's condition for any $t\in\mathbb{C}$ if and only if $n\leq q$. One can check that the same holds for $\bK=\bR$. For $n>q$ we therefore get $S_0(f_{nq})\subsetneq K_{\infty}(f_{nq})\neq \emptyset$.
\end{example}

\subsection{Questions and Conjecture}
We have seen above that the inclusions $B(f)\subset S_a(f)\cup f(\Sing f)\subset K_{\ity}(f)\cup f(\Sing f)$ hold for any $a\in\bR^n$. On other hand, one can have $S_{a}(f)\neq S_b(f)$, as shown in the following example for which we skip the computations: 

\begin{example}\label{ex:1} Let $f\colon\bR^2\to\bR$, $f(x,y)=y(x^2y^2+3xy+3)$. Then $S_{(0,0)}(f)=\emptyset$ and $S_{(0,1)}(f)\neq \emptyset$.

In particular, for the polynomial $g(x,y) := f(x,y-1)$ we have that $\emptyset\neq S_0(g) \not\subset \Bi(g)= \emptyset$ which is a real counterexample to a conjecture in \cite[p. 686, point 3]{NZ}.  A new conjecture can be stated as below.
\end{example}

These facts support the following natural questions:

\begin{itemize}
 \item[(i)] Is there a minimal set $S_{a}(f)$ in the collection $\{S_b(f),b\in\bR^n\}$? 
\item[(ii)]   If (i) is true, does the minimality hold for some open dense subset $a\in\bR^n$? 
\end{itemize}

The following conjecture seems also natural:
\begin{conjecture}  $\Bi(f) = \Si(f)$.
\end{conjecture}

%Compute the exemple by Paunescu-Zaharia to see if there exist some $a$ such that $0 \not\in S_a(f)$.

%%%%%%%%%%%%%%%%%
%%%%%%%%%%%%%%%%%%%%%

\section{Detecting bifurcation values at infinity by parametrized curves}\label{s:detec-S}

\subsection{Bounding the set of asymptotic critical values $\Si (f)$}\label{ss:detec-S}

Let $f\colon\bR^n\to\bR$ be a polynomial function of degree $d \ge 2$. 
 If $t_0\in \Si (f)$ then, by Proposition \ref{p:gen-a} and Definition \ref{d:rho-reg}, there exists an open dense set of points $a\in \bR^n$ such that $M_a(f)$ is a real affine algebraic set of dimension 1 and there exists a real asymptotic branch  $\Gamma\subset M_a(f)$  such that $\lim_{x\in \Gamma, \| x\| \to \infty} f(x)= t_0$.  Therefore such branches $\Gamma$ detect all the 
values in $\Si (f)$. 

Moreover,   $\# S_a(f)$ is precisely the number of finite values taken by $f$ when restricted to branches at infinity of  $M_a(f)$. It follows  that $\# \Si(f)$ is majorated by the number of branches at infinity.

This interpretation yields the same bound for the number of bifurcation values at infinity $\# \Bi(f)$
as the bound found by Jelonek and Kurdyka for $\# K_\ity(f)$.

\begin{proposition}\label{p:main} \cite[Corollary 1.2]{JK-crelle}

Let $f\colon\bR^n\to\bR$ be a polynomial function of degree $d \ge 2$.
The number $\# \Si(f)$ of asymptotical critical values is majorated by $d^{n-1} -1$.
\end{proposition}

\begin{proof}
 We may assume that $M_a(f)$ is not empty for $a\in \Omega_f$. By the proof of Lemma \ref{p:gen-a},  $M_a(f)$ is a real curve which is a local complete intersection defined by $n-1$ equations as in \eqref{eq:m-j}. 
Let us consider the closure $\overline{M_a(f)}$ in $\bP^n_{\bR}$.

 The complexification $M_a(f)^\bC$ of $M_a(f)$ is defined locally as the complex solutions of the same system of equations, hence it is a complex curve of degree equal to $d^{n-1}$. Let us denote by $\overline{M_a(f)^\bC}$ its closure in $\bP^n_\bC$, and notice that this is a complex curve of the same degree $d^{n-1}$. 

The hyperplane at infinity $H^\ity := \bP^n_\bC \m \bC^n$ intersects $\overline{M_a(f)^\bC}$ at finitely many points and by Bezout's theorem this number is bounded by the degree $d^{n-1}$. If $B(f) \m f(\Sing f) \not= \emptyset$  then the restriction of $f_\bC$ to $M_a(f)^\bC$ cannot be bounded, since there is at least one branch on which the holomorphic function $f$ is non-constant. Hence we 
can get at most $d^{n-1} -1$ finite values as limits of $f$.  
\end{proof}

%%%%%%%%%%%%%%%%%%%%%%%%%%%%%%%%%%%%%%%%%%%%%%%%%%%%%%%%%%%%

\subsection{Branches at infinity of the Milnor set and asymptotic parametrisations}\label{ss:param} 

 Let $f\colon\bR^n\to\bR$ be a polynomial function of degree $d \ge 2$. 
 If $t_0\in \Si (f)$ then, by Proposition \ref{p:gen-a} and Definition \ref{d:rho-reg}, there exists an open dense set of points $a\in \bR^n$ such that $M_a(f)$ is a real affine algebraic set of dimension 1 and moreover, there exists a real asymptotic branch  $\Gamma\subset M_a(f)$  such that $\lim_{x\in \Gamma, \| x\| \to \infty} f(x)= t_0$.  

 Let us describe briefly following the ideas of \cite{JK} how to produce real parametrisations of the branches at infinity of $M_a(f)$.  Let $M_a(f)^\bC$ denote its complexification  defined by the same equations. We recall that this is a locally complete intersection and has degree $D := d^{n-1}$.
One may construct Puiseux parametrisations for the branches at infinity $\Gamma_\bC$ of $M_a(f)^\bC$ like in \cite[Lemma 3.1]{JK} $\gamma_\bC(t) = \sum_{-\ity \le i \le D} c_i t^i$,
where $t\in \bC$ and $| t| > R$ for some radius $R \gg 1$. The bound to the right is $D = d^{n-1}$.
 
 Since each real branch at infinity $\Gamma$ of $M_a(f)$ is contained in some complex branch of $\Gamma_\bC$ of $M_a(f)^\bC$ then, following Milnor's procedure \cite[pag. 28-29]{Mi}, one may further construct\footnote{this part of our construction is different from the procedure given in \cite[Lemma 3.2]{JK}.} a real parametrisation  
 $\gamma(t)= \sum_{-\ity \le i \le D} a_i t^i$ of $\Gamma$ with the same bound $D=d^{n-1}$ to the right, where $t\in \bR$ with $| t| > R$ and the same radius $R \gg 1$.

Next we truncate the parametrisation $\gamma(t)$ at the left to the bound  $- D'$, where $D' := (d-1)D$ and $d$ is the degree of $f$. This 
 truncation $\hat \gamma(t) = \sum_{-(d-1)D \le i \le D} a_i t^i$ does not verify anymore the equations of  $M_a(f)$. Nevertheless we can show the following:
\begin{proposition}\label{p:trunc}
  The above defined truncation $\hat \gamma(t)$ has the following properties:
\begin{enumerate} 
\rm \item\label{eq:lim1} $\lim_{t \to \infty}  \|\hat \gamma(t)\|  = \lim_{t \to \infty}  \|\gamma(t)\|  = \infty. $ 
\rm \item\label{eq:lim2} $\lim_{t \to \infty} f(\hat \gamma(t))= \lim_{t \to \infty} f(\gamma(t))= t_0.$ 
\rm \item\label{eq:lim3} $\lim_{t\to \infty} \frac{\partial f }{  \partial x_i }(\hat \gamma(t))  =  \lim_{t\to \infty} \frac{\partial f }{  \partial x_i }(\gamma(t)) = 0, \mbox{ for any } i.$
\rm \item\label{eq:lim4} $\lim_{t\to \infty}x_j \frac{\partial f }{  \partial x_i }(\hat \gamma(t))  =  \lim_{t\to \infty}x_j \frac{\partial f }{  \partial x_i }(\gamma(t)) = 0, \mbox{ for any } i, j.$
\end{enumerate}
\end{proposition}
\begin{proof}
\noindent
(a) is obvious from the definitions.

\noindent
(b) follows from \cite[Lemma 3.3]{JK} since the degree of $f$ is $d$.

 \noindent
(c) and (d). We claim  that the limits involving $\gamma(t)$ are zero. If this is true,  then the same will follow for the truncation $\hat \gamma(t)$ by using \cite[Lemma 3.3]{JK} since the degrees of $\frac{\partial f }{ \partial x_i }$ and of $x_j \frac{\partial f}{\partial x_i}$
are both $\le d$. 

In order to prove our claim, we use our key result Theorem \ref{p:rho-rab} for $p=1$, as follows. Since in this case $\nu( \rD f(x))=\|\grad f(x))\|$, the equality $\lim_{t\to 0}\|\gamma(t)\| \|(\frac{\partial f}{\partial x_1}(\gamma(t)),\ldots,\frac{\partial f}{\partial x_n}(\gamma(t)))\|=0$ provided by Theorem \ref{p:rho-rab} clearly implies $\lim_{t\to \infty}x_j \frac{\partial f }{  \partial x_i }(\gamma(t)) = 0$ and 
$\lim_{t\to \infty} \frac{\partial f }{  \partial x_i }(\gamma(t)) = 0$.

 This completes our proof.
\end{proof}

%%%%%%%%%%%%%%%%%%%%%%%%%%%%%%%%%%%%%%%%%%%%%%%%%%%%%%%%%%%%

\subsection{The arc space}\label{ss:reaching}

Jelonek and Kurdyka \cite{JK} found an algorithm for reaching the values of $K_\ity(f)$ by parametrized curves with bounded expansion. They construct a finitely dimensional space of such rational arcs (cf \cite[Definition 6.9]{JK}). Based on our framework we construct a new space of rational arcs of considerably lower dimension. We shall see that this detects at least the bifurcation set $\Si(f)$ without necessarily detecting all the values in $K_\ity(f)$. 

 We consider here the following space of arcs associated to the real polynomial $f : \bR^n \to \bR$ of degree $d$:
\begin{equation}
\Arc(f)=\left\{ \xi(t) =\sum_{-(d-1)d^{n-1}\le k \le d^{n-1}} a_kt^k, a_k\in\bR^n\right\}
\end{equation}
which is isomorphic to $\bR^{n(1+ d^{n})}$. 

\begin{remark}
The dimension of the arc space $AV(f)$ constructed  in \cite{JK} is $n(2+ d(d+1)^n(d^n+2)^{n-1})$, while ours is $\dim \Arc(f) =n(1+ d^{n})$.
\end{remark}

\medskip

In a similar manner as in \cite[Definition 6.10]{JK}, we  define the  \emph{asymptotic variety of arcs}, $\Arc_{\ity}(f)\subset \Arc(f)$, as the algebraic subset  of the rational arcs $\xi(t)\in \Arc(f)$ such that: 
\begin{enumerate}
\item[(a')] $\sum_{k>0}\sum_{j=1}^{n}a_{kj}^2=1,$ where $a_k=(a_{k1},\ldots,a_{kn})$. 
\item[(b')] $f(\xi(t))= b_0 +\sum_{k=1}^{\infty} b_k t^{-k},$ where $b_0, b_k\in\bR$.
\item[(c')] $\frac{\partial f }{  \partial x_i }(\xi(t))  = \sum_{k=1}^{\infty} c_{ik}t^{-k},$ for any $i,$ where $c_{ik}\in\bR$.
\item[(d')] $x_j \frac{\partial f }{  \partial x_i }(\xi(t))  = \sum_{k=1}^{\infty} d_{ijk}t^{-k},$ for any $i, j,$ where $d_{ijk}\in\bR$.
\end{enumerate}
The conditions (a')--(d') for $\xi$ are equivalent with the corresponding properties (a)--(d) applied to $\xi$ insted of $\hat \gamma$. For instance the first equivalence 
follows by renormalising the coefficients.

Let $b_0 \colon \Arc_{\ity}(f) \to \bR,$ $b_0(\xi(t))= \lim_{t\to\ity}f(\xi(t))$.

\begin{theorem}\label{t:main} 
  $\Si(f) \subset b_0(\Arc_{\ity}(f))\subset K_{\ity}(f)$.
\end{theorem}

\begin{proof}
The inclusion  $b_0(\Arc_{\ity}(f))\subset K_{\ity}(f)$ is a direct consequence of the definitions of $\Arc_{\ity}(f)$ and $K_{\infty}(f)$ since properties (a'), (b') and (d') characterise the values $b_0\in K_{\infty}(f)$.

Let us show the first inclusion.
If $b_0\in\Si(f)$ then there is some $a\in \bR^n$ and there exists a path 
 $\gamma(t) \in M_a(f)$, such that $\lim_{t \to \infty} f(\gamma(t))= b_0$.
Then Theorem \ref{p:rho-rab} shows that $\gamma$ verifies the conditions (a)--(d) of Proposition \ref{p:trunc}. 
Moreover, by the same Proposition \ref{p:trunc} one has that the truncation $\hat\gamma$ defined at \S \ref{ss:param} verifies the same properties, hence the equivalent conditions (a')--(d') too. 
 This ends our proof.
\end{proof}

%The problem is now to be able to find the coefficients of such finite parametrisations $\hat \gamma(t)$. Let us show that we can use here the construction by Jelonek and Kurdyka \cite{JK}.

%%%%%%%%%%%%%%%%%%%%%%%%%%%%%%%%%%%%%%%%%
%%%%%%%%%%%%%%%%%%%%%%%%%%%%%%%%%%%%%%%%%%%%%%%%%%%%%%%%%%%%%%%%%%%%%%%%%%%%%%

%%%%%%%%%%%%%%%%%%%%%%%%%%%%%%%%%%%%%%%%%%%%%%%%%

%\bibliographystyle{alpha}

%%%%%%%%%%%%%%%%%%%%%%%%%%%%%%%%%%%%%%%%%%%%%%%%%%%%%%%

\end{document}